\documentclass[oneside, 10pt]{amsart}
\usepackage{amsmath, amsfonts,amsthm,times,graphics,amssymb, mathrsfs,yhmath}

\DeclareMathOperator*{\var}{Var }
\usepackage[active]{srcltx}
 \makeatletter
\renewcommand*\subjclass[2][2000]{%
  \def\@subjclass{#2}%
  \@ifundefined{subjclassname@#1}{%
    \ClassWarning{\@classname}{Unknown edition (#1) of Mathematics
      Subject Classification; using '1991'.}%
  }{%
    \@xp\let\@xp\subjclassname\csname subjclassname@#1\endcsname
  }%
}
 \makeatother
\newtheorem{theorem}{Theorem}[section]
\newtheorem{lemma}[theorem]{Lemma}
\newtheorem{corollary}[theorem]{Corollary}
\newtheorem{proposition}[theorem]{Proposition}
\theoremstyle{definition}
\newtheorem{definition}[theorem]{Definition}
\newtheorem{example}[theorem]{Example}
\newtheorem{remark}[theorem]{Remark}
\numberwithin{equation}{section}

 \makeatletter
\renewcommand*\subjclass[2][2000]{%
  \def\@subjclass{#2}%
  \@ifundefined{subjclassname@#1}{%
    \ClassWarning{\@classname}{Unknown edition (#1) of Mathematics
      Subject Classification; using '1991'.}%
  }{%
    \@xp\let\@xp\subjclassname\csname subjclassname@#1\endcsname
  }%
}

\renewcommand{\Re}{{\rm Re}}       
\renewcommand{\Im}{{\rm Im}}       
\def\NABLA#1{{\mathop{\nabla\kern-.5ex\lower1ex\hbox{$#1$}}}}
\def\Nabla#1{\nabla\kern-.5ex{}_{#1}}
\def\Tabla#1{\Tilde\nabla\kern-.5ex{}_{#1}}
\renewcommand{\Tilde}{\widetilde}

\makeatother
\begin{document}

\title{Charatheodory and Smirnov type theorem for harmonic mappings}


\author{David Kalaj}
\address{Faculty of Natural Sciences and Mathematics, University of Montenegro, Cetinjski put b.b., 81000 Podgorica, Montenegro} \email{davidk@ac.me}

\author{Marijan Markovi\'c}
\address{Faculty of Natural Sciences and Mathematics, University of Montenegro, Cetinjski put b.b., 81000 Podgorica, Montenegro}
\email{marijanmmarkovic@gmail.com}

\author{Miodrag Mateljevi\'c}
\address{Faculty of Mathematics, University of Belgrade,  Studentski trg 16, 11000 Belgrade, Serbia}
\email{miodrag@matf.bg.ac.rs}

\begin{abstract} We prove a version of Smirnov type theorem and Charatheodory type theorem for
a harmonic homeomorphism of the unit disk onto a Jordan surface with
rectifiable boundary.  Further we establish the classical
isoperimetric inequality and Riesz--Zygmund inequality for Jordan
harmonic surfaces without any smoothness assumptions of the
boundary.
\end{abstract}

\maketitle

\section{Introduction}
Throughout this paper $n\ge 1$ will be an integer. By
$\left<\cdot,\cdot\right>$ and $|\cdot|$ are denoted the standard
inner product and Euclidean norm in the space $\mathbb {R}^n$. In
particular $\mathbb{C}^n= \mathbb{R}^{2n}$, where
$\mathbb{C}=\mathbb{R}^2$ is the complex plane. By
$\mathbb{U}=\{z=x+iy\in\mathbb{C}:|z|<1\}$ we denote the unit disk
and by $\mathbb{T}=\{\zeta\in \mathbb{C}:|\zeta|=1\}$ is denoted the
unit circle in the complex plane.

Let $f=(f^1,\dots,f^n):\mathbb{U}\to\mathbb {R}^n$ be a continuous
mapping defined in the unit disc having partial derivatives of first
order in $\mathbb{U}$. The formal derivative (Jacobian matrix) of
$f$ is defined by
$$\nabla f=\left(
\begin{array}{cc}
f^1_x &  f^1_y \\
\vdots &  \vdots \\
f^n_x &  f^n_y \\
\end{array}
\right).$$ Jacobian determinant of $\nabla f$ is defined by
$$J_f=\left(\det [\nabla f^T\cdot \nabla f]\right)^{1/2}=\sqrt{|f_x|^2|f_y|^2-\left<f_x,f_y\right>^2}.$$
A mapping $f=(f^1,\dots,f^n):\mathbb{U}\to\mathbb{R}^n$ is called
harmonic if $f^j,\ j=1,\dots,n$ are harmonic functions in
$\mathbb{U}$, that is if $f^j$ is twice differentiable and satisfies
the Laplace equation
$$\Delta f^j  \equiv  0,\ j=1,\dots,n.$$
Let
$$P(r,t)=\frac{1-r^2}{2\pi (1-2r \cos t+r^2)},\quad 0\le r<1,\ 0\le t\le 2\pi$$
denote the Poisson kernel for the disc $\mathbb{U}$. It is well
known that every bounded harmonic mapping
$f:\mathbb{U}\to\mathbb{R}^n$ has the representation as Poisson
integral
\begin{equation}\label{e:POISSON}
f(z)=P[F](z)=\int_0^{2\pi}P(r,t-\theta)F(e^{it})\ dt,\quad
z=re^{i\theta}\in\mathbb {U},
\end{equation}
where $F:\mathbb{T}\to\mathbb{R}^n$ is a measurable and bounded in
the unit circle.

A homeomorphic image of the unit circle $\mathbb{T}$ in $\mathbb
{R}^n$ is called a Jordan curve. A Jordan surface $\Sigma\subseteq
\mathbb{R}^n$ is a homeomorphic image of the closed unit disk, i.e.
$\Sigma=\Phi(\overline{\mathbb{U}})$, where $\Phi$ is a
homeomorphism. We say that $\Sigma$ is spanned by the Jordan curve
$\Gamma=\partial\Sigma= \Phi(\mathbb T)$. It will be always assumed
that $\Gamma$ is at least rectifiable; we denote by $|\Gamma|$ its
length. The surface $\Sigma\subseteq \mathbb {R}^n$ is regular if
$\Sigma= \tau(\mathbb{U})$, where $\tau=\tau(x,y)$ is a $C^1$
injective mapping, with positive Jacobian  $J_\tau$ in $\mathbb{U}$.
 Thus the tangent vectors
$\tau_x,\ \tau_y $ are linearly independent for $z=x+iy\in\mathbb
{U}$ or equivalently the Jacobian matrix $\nabla f$ has full rank
$2$ in $\mathbb{U}$. The mapping $\tau$ is called parametrization of
$\Sigma$. Certainly, it is not unique. The area $|\Sigma|$ of the
surface $\Sigma$ equals
$$ \left| \Sigma\right|  =  \int_\mathbb{U}  J_\tau \ dA,$$
where $dA(z)=dxdy$ is Lebesgue measure in the complex plane.

We call a Jordan surface $\Sigma\subseteq\mathbb{R}^n$ a
simple--connected harmonic surface if there exist a homeomorphic
harmonic mapping $\tau: \mathbb{U}\to\Sigma$ (it need not have a
homeomorphic extension to $\overline{\mathbb{U}}$). Let us point out
that $\tau$ need not be a regular parametrization of $\Sigma$, i.e.
the strict inequality
$$J_\tau=\sqrt{|\tau_x|^2 |\tau_y|^2-\left<\tau_x,\tau_y\right>^2}>0\ \text{in the disc}\ \mathbb{U} $$
need not hold except in the planar case (in view of Lewy's theorem,
see \cite{l}). In other words, we allow that the harmonic surfaces
have branch points, i.e. the points with zero Jacobian.

Together with this introduction, the paper contains two more
sections. In the second section it is proved that a harmonic mapping
of the unit disk onto a Jordan surface has BV extension onto the
boundary. In addition it is proved the Smirnov theorem for harmonic
mappings of the unit disk onto a Jordan surface which assert that,
the angular derivative of a harmonic homeomorphism $h$ belong to the
Hardy class $h^1(\mathbf U)$ if and only if the boundary of the
surface is rectifiable. In the third section it is proved the
isoperimetric inequality for harmonic surfaces. More precisely, if
$A$ is the area of a Jordan harmonic surface $\Sigma$ and $L$ is its
circumference, then there hold the inequality $4\pi A\le L^2$. This
results is not surprising, and it can be find in the literature in
various formulations, but we believe that our inequality contains
some new information regarding the isoperimetric inequality,
especially because it contains the optimal relaxing condition of
smoothness of boundary. We finish the paper by establishing the
Riesz-Zygmund inequality which implies that the perimeter of a
harmonic surface is bigger than two "diameters".

\section{Carath\'eodory and Smirnov theorem for harmonic mappings}
Recall that a real-valued (or more generally complex-valued)
function $f$ on the real line is said to be of bounded variation
($BV$ function) on a chosen interval $[a, b]\subset \mathbb{R}$ if
its total variation $V_a^b(f)$  is finite, i.e. $f\in BV([a,b])$ if
and only if $V_a^b(f) <\infty$. The graph of a function having this
property is well behaved in a precise sense. A characterization
states that the functions of bounded variation on a closed interval
are exactly those $f$ which can be written as a difference $g - h$,
where both $g$ and $h$ are bounded monotone: this result is known as
the Jordan decomposition. Moreover, if $f$ is absolutely continuous
on $[a,b]\subset \mathbb{R}$, then $V^a_b(f) = \int _a^b |f'(t)|\
dt$.
\begin{theorem}[Helly selection theorem, \cite{natanson}]\label{helly}
Let $(\varphi_n)$ be a sequence of uniformly bounded functions of
uniformly bounded variation on a segment $[a,b]$. Then there exists
a subsequence $(\varphi_{n_k})\subseteq(\varphi_n)$ such that
$\varphi_{n_k}(x)\rightarrow\varphi(x)$ for every $x\in [a,b]$ and
$\varphi$ is of bounded variation. Moreover if all of $\varphi_n$
are monotone increasing (or decreasing), then so is $\varphi$.
\end{theorem}

\begin{lemma}\label{vaz} Let $f:\mathbb{U} \rightarrow \Sigma$ be a
harmonic homeomorphism of the unit disk onto a Jordan surface
$\Sigma$ with rectifiable boundary $\Gamma$. Then there exists a
function $\phi: \mathbb{T}\rightarrow \Gamma$ with bounded variation
and with at most countable set of points of discontinuity, where it
has the left and the right limit such that $f = P[\phi]$.
\end{lemma}

\begin{proof} Let $\Phi:\mathbb{U}\to\Sigma$ be a homeomorphism
which has extension on $\overline{\mathbb {U}}$ and let
$g=\Phi^{-1}$ be a homeomorphism of $\Sigma$ onto $\mathbb{U}$. The
function $F = g\circ f: \mathbb{U} \rightarrow \mathbb{U}$ is also a
homomorphism. Let $U_n = \{z: |z| < \frac{n-1}n\}$ and $\Delta_n =
F^{-1}(U_n)$ and let $g_n$ be a Riemann conformal mapping of
$\mathbb{U}$ onto the domain $\Delta_n$ such that $g_n(0)=0$ and
$g_n'(0)>0$. Assume w.l.g. that $0\in \Delta_n$. Then the function
$$F_n=\frac n{n-1} (F\circ g_n) = \frac n{n-1} (g \circ f \circ g_n ): \overline{\mathbb{U}} \rightarrow \overline{\mathbb{U}}$$
is a homeomorphism.

Then $\varphi_n(e^{i\theta})=e^{i\phi_n(\theta)}$ such that
$\phi_n(\theta)$ is a monotone function. Let $\varphi_n =
F_n|_\mathbb{T} $ and assume that $(\varphi_{n_k})$ is a convergent
subsequence $(\varphi_n)$ provided to us by Lemma \ref{helly}. Let
$\varphi_0 = \lim \varphi_{n_k}$. Then $\varphi_0$ is monotone.
Therefore
$$\frac{n_k}{n_k-1}(g\circ f\circ g_{n_k})|_\mathbb {T} \to \varphi_0.$$
It follows that
$$\lim_{k\to\infty}(f\circ g_{n_k})(e^{i\theta}) = \Phi(\varphi_0(e^{i\theta})) \mbox{ for every } \theta$$
because  $g=\Phi^{-1}$ is a homeomorphism $\overline{\Sigma}$ onto
$\overline {\mathbb{U}}$. Since $\Gamma$ is a rectifiable curve by
Scheeffer's theorem (\cite{acta}), the function $\Phi$ has bounded
variation in $\mathbb T$. Since $\varphi_0$ is monotone, it follows
that the mapping $\phi(e^{i\theta}) = \Phi(\varphi_0(e^{i\theta}))$
has bounded variation.

Since $\phi_k=(f\circ g_{n_k})|_\mathbb{T}$ is continuous and
$f\circ g_{n_k}$ is harmonic, according to Lebesgue Dominated
Convergence Theorem, because $g^{-1}\circ\varphi_0$ is bounded we
obtain
$$\lim_{k\to \infty}f \circ g_{n_k}=\lim_{k\to \infty}P[\phi_k] = P[\Phi \circ \varphi_0].$$

It follows that the sequence $g_{n_k}$ converges. Let
$g_0(z)=\lim_{k\to\infty} g_{n_k}(z)$. Since $g_0$ is a conformal
mapping of the unit disk onto itself such that $g_0(0)=0$ and
$g_0'(0)>0$, it follows that $g_0=id$. Therefore $f=P[\phi]$, where
$\phi=\Phi\circ\varphi_0$. Since $\Phi$ is continuous and
$\varphi_0$ is monotone, the mapping $\phi$ is continuous except in
a countable set of points $X$ where it has the left and the right
limit.
\end{proof}

The following two propositions are well-known. For the first one see
e.g. \cite[Section 1.4]{duren}.

\begin{proposition}\label{seg} Let $f:\mathbb{U}\rightarrow\Sigma$ be
a harmonic mapping of the unit disk onto the Jordan surface
$\Sigma$. If $f = P[\phi]$ and if for some $\theta_0\in [0,2\pi]$
holds
$$\lim_{\theta \uparrow \theta_0}\phi(e^{i\theta})=A_0$$
and
$$\lim_{\theta \downarrow \theta_0}\phi(e^{i\theta})=B_0,$$
then for $\lambda\in [-1,1]$ and a Jordan arc
$\Gamma_\lambda(s)\subseteq \mathbb U$, $0\le s < 1$ emanating at
$\Gamma_\lambda(1) = e^{i\theta}$ and forming the angle
$-\frac{\pi\lambda}{2}$ with $e^{i\theta}$ we have:
$$\lim_{s\to 1^-}f(\Gamma_\lambda(s)) = \frac 12(1-\lambda)A_0+\frac 12(1+\lambda)B_0.$$
\end{proposition}

\begin{proposition} If $f = P[\phi]$ and if $\phi$ is continuous at $\zeta
\in\mathbb T$, then $f$ has continuous extension on
$\zeta\cup\mathbb{U}$.
\end{proposition}

\begin{definition} At any point $\zeta$ of the unit circle $\mathbb T$,
the cluster set, $C_{\mathbb U} (f,\zeta)$, is defined as follows:
$\alpha\in C_{\mathbb U}(f,\zeta)$ if there exists a sequence $(z_n)
\subseteq \mathbb {U}$ such that $\lim_{n\to \infty} z_n = \zeta$
while $\lim_{n\to \infty} f(z_n) = \alpha$. It is known that for any
$\zeta$ the cluster set $C_{\mathbb U}(f,\zeta)$ is nonempty and
closed.
\end{definition}

\begin{theorem}\label{cont} Let $f:\mathbb{U}\rightarrow\Sigma$ be a
harmonic homeomorphism of the unit disk onto a Jordan surface with
rectifiable boundary $\Gamma$. Then
\begin{enumerate}

\item Then there exists a function
$\phi : \mathbb{T} \rightarrow \Gamma $ with bounded variation and
with at most countable set of  points of discontinuity, where it has
the left and the right limit such that $f=P[\phi]$.

\item  If the boundary $\Gamma$ of $\Sigma$ does not contain any segment,
then $f$ has a continuous extension up to the boundary.

\item If $e^{is}$ is a point of discontinuity of $\phi$, then there exist
$A_0=\lim_{t\uparrow  s^-} \phi(t), \ \ B_0=\lim_{t\downarrow
s}\phi(t)$ and $$C_{\mathbb U}(f,e^{is})=[A_0,B_0]\subseteq
\Gamma.$$
\end{enumerate}
\end{theorem}

\begin{proof} The item (1) is contained  in Lemma~\ref{vaz}.

Take $\theta_0 \in[0,2\pi]$. From item (1), there exist the left and
the right boundary values of $\psi$ at $\theta_0$. Let $\lim_{\theta
\uparrow \theta_0} f(e^{i\theta})=A_0$ and $\lim_{\theta \downarrow
\theta_0} f(e^{i\theta})=B_0$. For $R>0$ and for $-1\le \lambda \le
1$ let
$$z_R=e^{i\theta_0}\frac{Re^{i\lambda \frac{\pi}2}-1} {Re^{i\lambda \frac{\pi}2}+1}.$$
Then $z_R \to e^{i\theta_0}$ as $R \to \infty$ and the angle between
tangent of $\Gamma_\lambda=\{z_R: 1\le R\le \infty\}$ at $R=\infty$
and the point $e^{i\theta_0}$ is equal to $-\lambda\pi/2$. In view
of Lemma~\ref{seg} we have
$$\lim_{R \to \infty}f(z_R)=\frac 12(1-\lambda)A_0 +\frac 12(1+\lambda)B_0.$$
It follows that $[A_0,B_0]\subseteq C_{\mathbb{U}}(f,
e^{i\theta_0})$. Since $f:\mathbb {U}\to\Sigma$ is a homeomorphism,
it follows that $C_{\mathbb U}(f,e^{i\theta_0})\subseteq\Gamma$.
Therefore $[A_0,B_0] \subseteq\Gamma$. If $\Gamma$ do not contains
any segment then $A_0=B_0$, i.e. $\phi$ is continuous at
$e^{i\theta}$. This proves the item (2). To finish the proof of (3)
we need to show that $C_{\mathbb {U}}(f, e^{i\theta_0})\subseteq
[A_0,B_0]$. Let $z_n\to e^{i\theta_0}$ and assume that $X=\lim_{n\to
\infty} f(z_n)$. Since $z_n$ is a bounded sequence, it exists a
Jordan arc in $\mathbb{U}$ emanating at $e^{i\theta_0}$, forming the
angle $-\lambda_0\pi/2,\ \lambda_0\in[-1,1]$, with $e^{i\theta_0}$
and containing a subsequence $z_{n_k}$ of $z_n$. Thus
$$X=\lim_{n\to \infty} f(z_{n_k}) = \frac 12(1-\lambda_0)A_0 +\frac 12(1+\lambda_0)B_0\in [A_0,B_0].$$
\end{proof}
The previous theorem may be considered as Carath\'eodory theorem for
harmonic surfaces.
\begin{example}\cite{duren} Assume that $m>2$ is an integer and $0=\theta_0
< \theta_1 < \cdots < \theta_m < \theta_{m+1} = 2\pi$ and define
$$\varphi=\sum_{n=0}^m\theta_n\chi_{[\theta_n,\theta_{n+1}]}.$$
Then $f=P[e^{i\varphi(\theta)}]$ is a harmonic diffeomorphism of the
unit disk onto a Jordan domain inside the polygonal line with
vertices $e^{i\theta_l},\ l=0,1,\cdots,m$.
\end{example}

\begin{lemma}\label{gjat} Assume $\Sigma\subset\mathbb{R}^n$ is a Jordan
surface spanning a rectifiable curve $\Gamma$, parametrized by
harmonic coordinates $f$. Let $0<r<1$ and $\Gamma_r = f(r \mathbb
T)$. Then $|\Gamma_r|$ is increasing and
\begin{equation}\label{impo}
|\Gamma_r|\le |\Gamma|.
\end{equation}
\end{lemma}

\begin{proof} The total variation of a function $f\in BV$ is
$$\displaystyle \var_{\mathbb T_r} f = \sup \sum_{i=1}^m |f(re^{i s_{k+1}})-f(re^{i s_k})|,$$
where $\sup$ is taken for all finite divison of the unit circle
$\mathbb T$.

Put $z=re^{it}$. We proved that $f=P[F]$, where $F\in BV$. Then by
\eqref{e:POISSON}, using integration by parts, it follows that
$f_\tau$ equals the Poisson-Stieltjes integral of $dF$:
\[\begin{split}
f_\tau(re^{i\tau}) &=\int_0^{2\pi}\partial_\tau
P(r,\tau-t)F(t)dt\\&=-\int_0^{2\pi}\partial_t
P(r,\tau-t)F(t)dt\\&=-P(r,\tau-t)F(t)\big|_{t=0}^{2\pi}+\int_0^{2\pi}
P(r,\tau-t)dF(t) \\&= \int_0^{2\pi}P(r,\tau-t)dF(t).
\end{split}\]
Thus
\[\begin{split}\var_{\mathbf T_r} \partial_t f&=\int_{-\pi}^\pi
|\partial_t f(re^{it})| dt\\&=\int_{0}^{2\pi}
\left|\int_0^{2\pi}P(r,\tau-t)dF(\tau)\right| dt\\&\le
 \left|\int_0^{2\pi}|dF(\tau)|\right|\int_{0}^{2\pi}P(r,\tau-t) dt
\\&=
 \var_{0 \le t\le
2\pi}|F(t)|=|\Gamma|.
\end{split}\]
\end{proof}

\begin{lemma}\label{le.rect.rn} Let $f:\mathbb{U}\rightarrow\Sigma
\subseteq\mathbb{R}^n$ be a homeomorphism onto the Jordan surface
$\Sigma$ bounded by rectifiable curve $\Gamma$. Suppose that $f$ has
continuous extension on the $\Gamma\setminus E$, where $E$ is a
countable union of segments of $\Gamma$ (if there is any). Further,
let curves $\Gamma_r,\ 0<r<1$ defined by $f(re^{it}),\ 0\le t\le
2\pi$ be rectifiable. Then
$$\limsup_{r\rightarrow 1} |\Gamma_r|\ge |\Gamma|.$$
\end{lemma}

\begin{proof} Let $d(x,y) = |x-y|$ be the distance between points $x$
and $y$ in $\mathbb{R}^n$. Let $\varepsilon>0$ be fixed. There exist
points $\omega_0,\ \omega_1,\dots, \omega_n \in \Gamma$ such that
$$\sum_{j=0}^n  d(\omega_j, \omega_{j+1})> |\Gamma| - \varepsilon/ 2,$$
where we set $\omega_{n+1}=\omega_0$. We may suppose w.l.g. that
these points do not lie in $\Gamma\setminus E$.

Since $f$ has continuous extension on the boundary of $\Sigma$
without segments, we can find points $\zeta_j\in\mathbb{T}$ such
that $f(\zeta_j)= \omega_j$ for all $j=0,\ 1,\dots,n$. Let
$\omega'_j=f(r\zeta_j)\in\Gamma_r$. The distance between $r\zeta_j$
and $\zeta_j$ is $1-r$ for all $j$. Since $n$ is fixed and since $f$
has continuous extension on $\Gamma \setminus E$, there exist $r$
close enough to $1$ such that
$$s_n:=\sum_{j=0}^n d(\omega'_j , \omega_j) < \varepsilon/ 4$$
Using the triangle inequality
$$d( \omega_j,\omega_{j+1})\le  d(\omega_j,\omega'_j) + d( \omega'_j,\omega'_{j+1}) + d(\omega'_{j+1} , \omega_{j+1}),$$
we get
$$|\Gamma_r| \ge \sum_{j=0}^{n} d(\omega_j',\omega'_{j+1}) \ge \sum_{j=0}^n  d(\omega_j,\omega_{j+1}) - 2 s_n > |\Gamma| -\varepsilon.$$
Since $\varepsilon$ is an arbitrary positive number, it follows
$\lim\sup_{r\rightarrow 1} |\Gamma_r|\ge |\Gamma|$.
\end{proof}

Smirnov theorem for holomorphic functions can be generalized to
harmonic quasiconformal mappings (\cite{matkal}). The following
version of Smirnov only request harmonicity of a homeomorphism and
somehow is optimal.
\begin{theorem}\label{multth2} Let $f:{\mathbb U}\rightarrow \Sigma
\subset \mathbb {R}^n$ be a harmonic homeomorphism on the unit disk
onto the Jordan surface bouneded by the curve $\Gamma$ and let
$\Gamma_r,\ 0 < r <1$ be curves defined by $f(re^{it}),\ 0\le t \le
2\pi$. Then $\partial_t f\in h^1(\mathbb{U})$ if and only if
$\Gamma$ is rectifiable. In this settings, $|\Gamma_r|\rightarrow
|\Gamma|$ as $r\rightarrow 1$.
\end{theorem}
\begin{proof} If $\Gamma$ is rectifiable, according to Lemma~\ref{gjat}
we have $|\Gamma_r|\le |\Gamma|$ what means
$$\int_0^{2\pi} |\partial_t f(re^{it})|\ dt \le |\Gamma|.$$
Thus $\partial_t f\in h^1(\mathbb{U})$. On the other hand, if
$\partial_t f\in h^1(\mathbb{U})$, then $|\Gamma_r|$ is bounded and
according to previous lemma $|\Gamma|$ is finite. Since we have
harmonic parametrization, $|\Gamma_r|$ is an increasing sequence,
thus $\lim_{r\to 1} |\Gamma_r|\le |\Gamma|$, by Lemma \ref{gjat}.
Using lemma \ref{le.rect.rn} we have the reverse inequality. It
follows $\lim_{r\to 1} |\Gamma_r| = |\Gamma|$.
\end{proof}

In the settings of the previous theorem, in general, parametrization
for $\Gamma$ which is induced by $f$ is not always absolutely
continuous (or even continuous). In particular, if $n=2$ and $f$ is
holomorphic, then $f$ induce on $\Gamma$ an absolutely continuous
parametrization (this is Smirnov theorem). Thus there is difference
between harmonic and holomorphic concerning the property of absolute
continuity; see Proposition 2.1 in \cite{maboknez}.
\section{Some classical inequalities for harmonic surfaces-revisited }
Our first aim in this section is to establish the classical
isoperimetric inequality for harmonic surfaces with rectifiable
boundary and without any smoothness assumption on the boundary. We
expect that some of results we prove in this section are well-known,
but due to missing quick references, we include their proofs here.
\subsection{Gaussian curvature of a smooth surface}
The first fundamental form of a surface
$\Sigma\subseteq\mathbb{R}^n$ (not necessary a Jordan surface)
parametrized by a smooth mapping
$\tau(z)=(\tau_1(z),\dots,\tau_n(z)):\Omega\to \Sigma,\ z=x+iy$ is
given by
$$ds^2=Edx^2+2G dx dy+F dy^2$$
where $E = g_{11} = |\tau_x|^2,\ F = g_{12}
=\left<\tau_x,\tau_y\right>$ and $G = g_{22} = |\tau_y|^2$ satisfy
$E G - F^2>0$ on $\Omega$.

The Gaussian curvature is usually expressed as a function of the
first and second fundamental form. However for the surface which are
not embedded in $\mathbb {R}^3$ the second fundamental form is not
defined because it depends on Gauss normal, which is not defined in
a usual way in $\mathbb{R}^n,\ n\ge 4$. The Brioschi formula for the
Gaussian curvature gives us an opportunity to express it by
$$ K(x,y)=  \frac{\left| \begin{array}{ccc}
 -\frac 12 E_{yy} + F_{xy} - \frac 12 G_{xx} & \frac 12 E_x & \frac 12 F_x-\frac 12 G_x\\
 F_y -\frac 12 G_x& E & F \\
\frac 12 G_y & F & G \end{array} \right|
 - \left| \begin{array}{ccc}
 0 & \frac 12 E_y & \frac 12 G_x\\
\frac 12 E_y & E & F \\
\frac 12 G_x & F & G \end{array} \right|}{(EG -F^2)^2}.$$ This is
indeed an alternative formulation of the fundamental Gauss's Theorem
Egregium and consequently  the Gaussian curvature does not depend
whether the surface is embedded on $\mathbb{R}^3$ or in some other
Riemann manifold.

For three vectors $a=(a_1,\dots, a_n)$, $b=(b_1,\dots, b_n)$ and
$c=(c_1,\dots, c_n)$ we define the matrix
$$[a,b,c]:=\left(
\begin{array}{cccc}
a_1 & a_2 & \dots & a_n \\
b_1 & b_2 & \dots & b_n \\
c_1 & c_2 & \dots & c_n \\
\end{array}
\right).$$

\begin{lemma} Let $\Sigma$ be a surface in $\mathbb{R}^n$ with
parametrization $\tau = \tau(x,y) = (\tau_1,\dots,\tau_n)$ which is
enough smooth. The Gaussian curvature can be expressed as
\begin{equation}\label{formula}
K(x,y)=\frac{\det([\tau_{xx},\tau_x,\tau_y]\times[\tau_{yy},\tau_x,\tau_y]^T)
-\det([\tau_{xy},\tau_x,\tau_y]\times[\tau_{xy},\tau_x,\tau_y]^T)}
{(|\tau_x|^2|\tau_y|^2-\left<\tau_x,\tau_y\right>^2)^2}.
\end{equation}
\end{lemma}

\begin{remark} In standard expressions for Gaussian curvature, it
appears the third derivative of the parametrization. In formula
\eqref{formula} we have only the first and the second derivative
which is intrigue, but the proof depends on the third derivative of
$\tau$ as well and thus we should assume that the regularity of
$\tau$ is something more than class $C^2$.
\end{remark}

\begin{proof} First of all we have the equalities
$$E_y=2\left<\tau_{xy},\tau_x\right>, \  \  E_{yy}=2\left<\tau_{xyy},\tau_x\right>+2\left | \tau_{xy} \right |^2,$$
$$F_x=\left<\tau_{xx},\tau_y\right>+\left<\tau_x,\tau_{xy}\right>, \  \ F_{xy}=\left<\tau_{xxy},\tau_y\right>+\left<\tau_{xx},\tau_{yy}\right>+\left|\tau_{xy}\right|^2+\left<\tau_x,\tau_{xyy}\right>,$$
$$G_x=2\left<\tau_{xy},\tau_y\right>, \  \ G_{xx}=2\left<\tau_{xxy},\tau_y\right>+2\left|\tau_{xy}\right|^2$$
and
$$-\frac 12 E_{yy} + F_{xy} - \frac 12 G_{xx}=\left<\tau_{xx},\tau_{yy}\right>- \left|\tau_{xy}\right|^2.$$
Then
\[\begin{split}
\det([\tau_{xy},\tau_x,\tau_y]\times[\tau_{xy},\tau_x,\tau_y]^T) &=
\left| \begin{array}{ccc}
 |\tau_{xy}|^2 & \frac 12 E_y & \frac 12 G_x\\
\frac 12 E_y & E & F \\
\frac 12 G_x & F & G \end{array} \right|
\\& = \left| \begin{array}{ccc}
 |\tau_{xy}|^2 & 0 & 0\\
\frac 12 E_y & E & F \\
\frac 12 G_x & F & G \end{array} \right|
 + \left| \begin{array}{ccc}
 0 & \frac 12 E_y & \frac 12 G_x\\
\frac 12 E_y & E & F \\
\frac 12 G_x & F & G \end{array} \right|
\end{split}\]
and
\[\begin{split}
\det([\tau_{xx},\tau_x,\tau_y]&\times[\tau_{yy},\tau_x,\tau_y]^T)
\\& = \left| \begin{array}{ccc}
|\tau_{xy}|^2 -\frac 12 E_{yy} + F_{xy} - \frac 12 G_{xx} & \frac 12 E_x & \frac 12 F_x-\frac 12 G_x\\
 F_y -\frac 12 G_x& E & F \\
\frac 12 G_y & F & G \end{array} \right|
\\& = \left| \begin{array}{ccc}
|\tau_{xy}|^2 & 0 & 0 \\
 F_y -\frac 12 G_x& E & F \\
\frac 12 G_y & F & G \end{array} \right|
\\& + \left| \begin{array}{ccc}
 -\frac 12 E_{yy} + F_{xy} - \frac 12 G_{xx} & \frac 12 E_x & \frac 12 F_x-\frac 12 G_x\\
 F_y -\frac 12 G_x& E & F \\
\frac 12 G_y & F & G \end{array} \right|.
\end{split}\]
The equality of the lemma now follows from Brioschi formula for
Gaussian curvature.
\end{proof}

\begin{theorem}\label{kneg} If $\Sigma$ is a simple connected harmonic
surface which allows regular harmonic parametrization $\tau$, then
the Gaussian curvature of $\Sigma$ is nonpositive.
\end{theorem}

\begin{proof} Let $\Sigma$ be a simple connected harmonic surface with
regular harmonic parametrization $\tau$, that is, let $\Delta \tau
=(0,\dots,0)$. Since $\tau_{yy}=-\tau_{xx}$ we obtain
$$\det([\tau_{xx},\tau_x,\tau_y]\times[\tau_{yy},\tau_x,\tau_y]^T)
-\det([\tau_{xy},\tau_x,\tau_y]\times[\tau_{xy},\tau_x,\tau_y]^T)$$
$$=-\det([\tau_{xx},\tau_x,\tau_y]\times[\tau_{xx},\tau_x,\tau_y]^T)
-\det([\tau_{xy},\tau_x,\tau_y]\times[\tau_{xy},\tau_x,\tau_y]^T)\le
0,$$ because the corresponding matrices are symmetric. The previous
lemma implies that the Gauss curvature of $\Sigma$ is non-positive.
\end{proof}

Since the Gaussian curvature is an intrinsic invariant of the
surface, from Theorem \ref{kneg} we deduce the following result.

\begin{theorem}[Isoperimetric inequality for harmonic surfaces]\label{per}
If $\Sigma\subset \mathbb {R}^n$ is a Jordan harmonic surface with
rectifiable boundary $\Gamma$, then we have the classical
isoperimetric inequality
\begin{equation}\label{shif}
{4\pi} |\Sigma|\le |\Gamma|^2.
\end{equation}
\end{theorem}

\begin{proof} Let $\tau:\mathbb U\to \Sigma$ be a harmonic parametrization
of $\Sigma$. Since $\tau$ is not necessarily regular, as in
\cite{shiff}, let us perturb the surface $\Sigma$ in
$\mathbb{R}^{n+2}$ by taking for $\varepsilon>0$ and $0<r<1$ the
harmonic homeomorphism $\tau^\varepsilon_r(z) =(\tau(rz),\varepsilon
z)\in\mathbb {R}^{n+2},\ z\in \mathbb{U}$ we obtain a harmonic
parametrization  of a regular simple--connected harmonic surface
$\Sigma^\varepsilon_r = \tau^\varepsilon_r(
\mathbb{U})\subseteq\mathbb R^{n+2}$ with smooth boundary. Since the
Gauss curvature of $\Sigma^\varepsilon_r$ is non-positive, applying
the classical result, we obtain
\begin{equation}\label{sig}4\pi|\Sigma^\varepsilon_r|\le
|\Gamma^\varepsilon_r|^2.\end{equation} Letting first
$\varepsilon\to 0$ and then $r\to 1$, by the inequality \eqref{impo}
we obtain \eqref{shif}.

We  offer another proof of Theorem~\ref{per} by using the result of
Beeson \cite{beeson}, but this case we make use of
Theorem~\ref{multth2}. Since $\tau^\varepsilon_r$ converges to
$\tau$ and $|\Gamma^\varepsilon_r|$ converges to $|\Gamma|$, it
follows that $|\Sigma^\varepsilon_r|$ converges to $|\Sigma|$, and
in view of \eqref{sig} the inequality \eqref{shif} follows
immediately.
\end{proof}

\begin{remark} Theorem~\ref{per} can be considered as an variation of theorem of Shiffman \cite{shiff}. Namely Shiffman in
order to prove the isoperimetric inequality for harmonic surfaces
$\Sigma$ used the assumption that the harmonic parametrization
$\tau$ is a homeomorphism with $\tau|_{\mathbf T}\in \mathrm{BV}$.
Our proof shows that the condition $\tau|_{\mathbf T}\in {BV}$ is
somehow redundant, but we make a topological condition that $\Sigma$
is a Jordan surface.  We decide to present this inequality in this
paper, because it is not well-known. Additional motivation why we
consider this problem comes from the famous Courant book \cite{cou}
(see the proof of \cite[Theorem~3.7]{cou}), which has been published
some years after the paper of Shiffman. Indeed Courant proved for
$n=3$ the inequality
$$4 |\Sigma|\le |\Gamma|^2,$$
under the condition $\Sigma=\tau(\mathbb U)$, where $\tau$ is a
harmonic parametrization with absolutely continuous boundary data.

If we restrict ourselves to regular surfaces with smooth boundaries,
our Theorem~\ref{per} does not bring any new information, because it
is well-known the following fact,  the Riemann surface enjoys the
isoperimetric inequality (in compact smooth Jordan sub-surfaces) if
and only if the Gaussian curvature is non-positive (cf.
\cite{isak,hub,ose}) (This is a theorem of Beckenbach and Rad\'o).
 However we
believe that the Theorem~\ref{per} bring some new light on this
problem.
We strongly believe that the  Theorem~\ref{per} is well-known for
Jordan minimal surfaces and this particular case can be proved
without Theorem~\ref{cont}. Recall that Enneper-Weierstrass
parameterization
$$\tau(z)=(p_1(z),\dots, p_n(z)),\quad  z\in \mathbf U,$$
of a simple-connected minimal surface $\Sigma$ has harmonic
coordinates $p_j(z),\ j=1,\dots,n$ such that $p_j(z) = \Re
(a_j(z))$, where $a_j,\ j=1,\dots,n$ are analytic functions on the
unit disk satisfying the equation $\sum_{j=1}^n (a'_j(z))^2=0.$ 
\end{remark}

\subsection{Riesz-Zygmund inequality} The following is a classical
inequality.
\begin{proposition}[Riesz-Zygmund inequality]\cite[Theorem~6.1.7]{lib}\label{miro}
If $f\in h^1(\mathbb{U})$ is a harmonic function then
$$\int_{-1}^1|\partial_r f(r e^{is})|dr \le \frac 12 \int_{0}^{2\pi}|\partial_t f(e^{it})|dt.$$
The constant $1/2$ is the best possible.
\end{proposition}
As a corollary we have the next inequality.
\begin{corollary} Assume that $f$ is a harmonic diffeomorphism from
unit disc $\mathbb{U}$ onto a Jordan domain $\Omega$ with the
rectifiable boundary $\Gamma$ and let $d$ be an arbitrary diameter
of $\mathbb{U}$. Then, if by $|\cdot|$ we denote the corresponding
length, we have
$$ |f(d)| \le  |\Gamma |/2 .$$
\end{corollary}

Now we prove the following extension of Proposition \ref{miro}.

\begin{theorem}[Riesz-Zygmund inequality for harmonic surfaces]\label{para}
Assume $\Sigma\subseteq\mathbb{R}^n$ is a harmonic surface spanning
a rectifiable curve $\Gamma$ parametrized by harmonic coordinates.
Then for every $s\in [0,2\pi]$
\begin{equation}\label{ins}
\int_{-1}^1  |\partial_r\tau(re^{is}) | dt \le \frac 12
\int_0^{2\pi} |\partial_t \tau(e^{it})| dt.
\end{equation}
In other words, the length of the image of an arbitrary diameter $d$
of the unit disk under a harmonic parametrization $\tau$ is less
than one half of the perimeter of the surface $\Sigma$.
\end{theorem}

\begin{proof} Assume that $\tau$ are harmonic coordinates. Let $\tau=
(\Re(a_1),\dots, \Re (a_n))$, where $a_j,\ j=1,\dots, n$ are
analytic function in the unit disk. Then
$$\partial_t \tau + i r \partial_r \tau = (a_1',a_2',\dots, a_n')\subset \mathbb{C}^n$$
and thus $r \partial_r \tau$ is the harmonic conjugate of
$\partial_t\tau$. It follows that
\begin{equation}\label{pa}r
\partial_r \tau(re^{is})=\frac{1}{2\pi}\int_{-\pi}^\pi (\Im F[r e^{i t}])\partial_t \tau(e^{i (s-t)})dt,
\end{equation}
where $F(z)=2z/(1-z)$. As in the proof of \cite[Theorem~6.1.7]{lib}
we find out that
\begin{equation}\label{fub}
\int_{-1}^1|r^{-1}\Im F(r e^{i t})|dr = \pi
\end{equation}
for $0<|t|<\pi$. By Fubini's theorem, \eqref{pa} and \eqref{fub} we
obtain
\[\begin{split}
\int_{-1}^1 |\partial_r \tau(re^{is}) | dr & \le \frac
1{2\pi}\int_{0}^{2\pi} |\partial_t \tau(e^{it})|dt
\int_{-1}^1|r^{-1}\Im F(r e^{i t})|dr =\frac 12
\int_{0}^{2\pi}|\partial_t \tau(e^{it})|dt.
\end{split}\]
\end{proof}

\begin{remark} It is worth to notice the following important fact.
For a minimal surface $\Sigma$ over a domain in the complex plane,
every isothermal parametrization is a harmonic parametrization and
it coincides with Enneper--Weierstrass parametrization of the
minimal surface.
\end{remark}

Let $\Sigma\subseteq\mathbb{R}^n$ be a regular surface. For two
points $P,\ Q \in \Sigma$ we define the intrinsic distance as
follows
$$d_I(P,Q)=\inf_{c\in \mathfrak C}|c|,$$
where $\mathfrak {C}$ is the set of all Jordan arcs $c$ of $\Sigma$
with the length $|c|$ connecting $P$ and $Q$. It should be noted the
following fact, for close enough points $P$ and $Q$ it exists a
geodesic line $\gamma$ connecting $P$ and $Q$ such that $d_I(P,Q)
=|\gamma|$. We define the (geodesic) diameter of $\Sigma$ as
$$\mathrm{diam}(\Sigma^2)=\sup_{P, Q\in \Sigma}d_I(P,Q).$$

We can now deduce the following geometric application of Theorem
\ref{para}.

\begin{theorem}\label{mai} If $\Sigma\subseteq\mathbb {R}^n$ is an
arbitrary simply connected harmonic surface with rectifiable
boundary $\Gamma$ then:
\begin{equation}\label{dia}
\mathrm{diam}(\Sigma) \le  \frac{1}{2}|\Gamma|.
\end{equation}
The constant $1/2$ is the best possible even for minimal surfaces
lying over the unit disk.
\end{theorem}

\begin{proof} Without lost of generality, let $\tau : \mathbb{U}\to
\Sigma$ be regular harmonic parametrization of the surface $\Sigma$
(if not we can perturb surface in $\mathbb{R}^{n-2}$ as in the proof
of isoperimetric inequality). Let $P,\ Q\in \Sigma$. Then there
exist a conformal mapping $a$ of the unit disk $\mathbb{U}$ onto
itself such that $\tau(a(-x))=P$ and $\tau(a(x))=Q,\ 0<x\le 1$. Take
$\upsilon_\delta(z)=\tau\circ a(\delta z),\ x<\delta<1$. Then by
Theorem \ref{para} and relation \eqref{impo} we have
$$d_I(P,Q)\le \int_{-1}^1|\partial_r \upsilon_\delta(r)| dr
< \frac 12 \int_{0}^{2\pi}|\partial_t \upsilon_\delta(e^{it})|dt \le
\frac 12 |\gamma|.$$ By $d_I(P,Q) < |\gamma|/2$ we obtain
\eqref{dia}.

Show that the constant $1/2$ is sharp. Assume, as we may that $n=3$.
Let $d=[-e^{it},e^{it}]$ be an arbitrary diameter of the unit disk
and let
$$\tau(x,y)=(x,y,m(x+y))$$ where $m$ is a large constant. We can
express the perimeter of the minimal surface $\tau$ by Elliptic
integral of the second kind $E$ i.e.
$$|\gamma|=2 (E[\pi/4, -2 m^2] + E[(3 \pi)/4, -2
m^2]).$$ The length of $\tau(d)$ is $2 \sqrt{1 + m^2 + m^2 \sin{2
t}}$. The maximal diameter is attained for $t=\pi/4$ and is equal
$2\sqrt{1+2m^2}$. Then
$$\lim_{m\to\infty}\frac{2\sqrt{1+2m^2}}{2 (E[\pi/4, -2 m^2] + E[3
\pi/4, -2 m^2])}=\frac12.$$
\end{proof}


\begin{thebibliography}{99}



\bibitem{beeson}
\textsc{M. Beeson:} {\it On the Area of Harmonic Surfaces}, Proc.
Amer. Math. Soc. {\bf 69} (1978).







\bibitem{maboknez}
\textsc{ V. Bo\v zin, M. Mateljevi\'c, and M. Kne\v zevi\'c,} {\it
Quasiconformality of harmonic mappings between Jordan domains},
Filomat, {\bf 24} (2010), 111--124.
\bibitem{isak}
{\sc I. Chavel}: \emph{Riemmanian Geometry: A Modern Introduction}
(2nd edition), Cambridge Studies in Advanced Mathematics 98,
Cambridge University Press, Cambridge, 2006.
\bibitem{cou}
{\sc R. Courant:} {\it Dirichlet's principle, conformal mapping, and
minimal surfaces. With an appendix by M. Schiffer.} Reprint of the
1950 original. Springer-Verlag, New York-Heidelberg, 1977. xi+332
pp.
\bibitem{duren}
\textsc{P. Duren:} {\it Harmonic mappings in the plane.} Cambridge
University Press, 2004.







\bibitem{heng}
{\sc Hengartner, W., Schober, G.} : {\it Harmonic mappings with
given dilatation}, J. London Math. Soc. (2)  {\bf 33} (1986), no. 3,
473--483.
\bibitem{hub}  \textsc{A. Huber:} {\it On the isoperimetric inequality on surfaces of variable
Gaussian curvature}, Ann. of Math.(2). 60(1954), 237-247.






\bibitem{matkal}
\textsc{ D. Kalaj, M Mateljevic:} \emph{$(K, K')$-quasiconformal
Harmonic Mappings}, Potential Analysis, 2012, Volume 36, Number 1,
Pages 117-135.







\bibitem{l}
H. Lewy: {\it On the non-vanishing of the Jacobian in certain in
one-to-one mappings,} Bull. Amer. Math. Soc. 42. (1936), 689-692.

\bibitem{lozinski} \textsc{S. Lozinski:} \emph{On subharmonic functions and their application to
the theory of surfaces}, Izv. Akad. Nauk SSSR Ser. Mat., 8:4 (1944),
175--194.

\bibitem{mp1}
\textsc{M. Mateljevi\'c, M. Pavlovi\'c:} {\it Some inequalities of
isoperimetric type concerning analytic and subharmonic functions.}
Publ. Inst. Math. (Beograd) (N.S.)  50(64)  (1991), 123--130.





\bibitem{ose}
\textsc{R. Osserman:} {\it The isoperimetric inequality}, Bull.
Amer. Math. Soc. \textbf{84} (1978), no. 6, 1182--1238.

\bibitem{lib}
{\sc M. Pavlovi\'c:} {\it Introduction to function spaces on the
disk}, Matemati\v cki Institut SANU, Belgrade, 2004. vi+184 pp.
\bibitem{natanson}
{\sc Natanson, I. P.} : {\it Theory of Functions of a Real
Variable}, Ungar, New York, 1955 and 1960.



\bibitem{acta}
\textsc{L. Scheeffer} : {\it Allgemeine Untersuchungen \"uber
Rectification der Curven},  Acta Math. {\bf 5} (1885), pp. 49--82.

\bibitem{shiff}
{\sc M. Shiffman:} {\it On the isoperimetric inequality for saddle
surfaces with singularities}, Studies and Essays Presented to R.
Courant on his 60th Birthday, January 8, 1948, pp. 383--394.
Interscience Publishers, Inc., New York, 1948.

\bibitem{weil}
\textsc{A. Weil}: {\it Sur les surfaces a courbure negative}, C. R.
Acad. Sci. Paris {\bf 182} (1926), 1069--1071.

\end{thebibliography}
\end{document}